\documentclass{article}

\usepackage{amsmath,amsthm,amssymb}
\usepackage{enumitem}
\usepackage{mathtools}

\usepackage[a4paper, total={5.5in, 8in}]{geometry}

\usepackage[utf8]{inputenc}
\usepackage[english]{babel}
\usepackage{pgfplots}
\usepgfplotslibrary{fillbetween}
\pgfplotsset{compat=1.12}

\numberwithin{equation}{section}

\newtheorem{theorem}{Theorem}[section]
\newtheorem{thmletter}{Theorem}
\newtheorem{thmletterprime}{Theorem}

\newtheorem{proposition}[theorem]{Proposition}
\newtheorem{lemma}[theorem]{Lemma}
\newtheorem{corollary}[theorem]{Corollary}

\theoremstyle{definition}
\newtheorem{definition}[theorem]{Definition}
\newtheorem{remark}[theorem]{Remark}

\date{}

\newcommand\blfootnote[1]{%
	\begingroup
	\renewcommand\thefootnote{}\footnote{#1}%
	\addtocounter{footnote}{-1}%
	\endgroup
}

\def\N{\mathbb N}
\def\R{\mathbb R}
\def\C{\mathbb C}
\def\S{\mathbb S}
\def\Rn{\mathbb{R}^n}
\def\ff{\widehat{F}}

\title{Uniform convergence of Hankel transforms}

\author{A. Debernardi\footnote{E-mail: \texttt{adebernardipinos@gmail.com}} \\ 
	\small{Centre de Recerca Matem\`atica and Universitat Aut\`onoma de Barcelona}\\ \small{08193, Bellaterra, Barcelona, Spain}}

\begin{document}
	
\maketitle
\begin{flushleft}
	\small{AMS 2010 Primary subject classification: 42A38. Secondary: 26A48,  40A10, 47G10}
\end{flushleft}

\section*{Abstract}
We investigate necessary and/or sufficient conditions for the pointwise and uniform convergence of the weighted Hankel transforms
$$
\mathcal{L}^{\alpha}_{\nu,\mu}f(r)=r^\mu\int_0^\infty (rt)^\nu f(t) j_\alpha(rt)\, dt, \quad \alpha\geq -1/2, \quad r\geq 0,
$$
where $\nu,\mu\in \R$ are such that $0\leq \mu+\nu\leq \alpha+3/2$. We subdivide these transforms into two classes in such a way that the uniform convergence criteria is remarkably different on each class. In more detail, we have the transforms satisfying $\mu+\nu=0$ (such as the classical Hankel transform), that generalize the cosine transform, and those satisfying $0<\mu+\nu\leq \alpha+3/2$, generalizing the sine transform.
\blfootnote{This research was partially funded by the CERCA Programme of the Generalitat de Catalunya, Centre de Recerca Matem\`atica, and the grant MTM2014--59174--P.}
	
\section{Introduction}
While studying which conditions are necessary and sufficient to guarantee uniform convergence of the Fourier transform 
\begin{equation}
	\label{ftrans}
	\ff(y) = \int_{\Rn} F(x) e^{i x\cdot y}\, dx,
\end{equation}
one encounters the necessity to impose restrictions on $F$. It is clear that the uniform convergence of \eqref{ftrans} follows from the condition $F\in L^1(\Rn)$. However, the latter condition is too restrictive and sometimes not even necessary (see, e.g., \cite{CJ,DLT,LS,ZS}).

It is known that the Fourier transform of a radial function $F(x)=f_0(|x|)$ is also a radial function \cite{SWfourier} given by
\begin{equation}
	\label{fourierradial}
	\widehat{F}(y)  = |\S^{n-1}|\int_{0}^{\infty} t^{n-1} f_0(t) j_\alpha(2\pi |y|t)\, dt, \quad \alpha = \frac{n}{2}-1, \quad n\geq 2,
\end{equation}
where $|\S^{n-1}|$ denotes the area of the unit sphere $\S^{n-1} =\{x\in \Rn : |x|=1\}$, $j_\alpha(z)$ is the normalized Bessel function
\begin{equation}
	\label{besselsimple}
	j_\alpha(z)=\Gamma(\alpha + 1) \bigg(\frac{z}{2}\bigg)^{-\alpha} J_\alpha(z),
\end{equation}
and $J_\alpha(z)$ is the classical Bessel function of the first kind of order $\alpha$. Basic properties of these functions are discussed in Section~\ref{sec2}.

For every $\alpha\geq -1/2$, the \emph{Hankel transform of order $\alpha$} of a function $f:\R_+\to \C$ is defined as
\begin{equation}\label{hankel}
	H_\alpha f (r) =  \frac{2\pi^{\alpha+1}}{\Gamma(\alpha+1)}\int_{0}^{\infty} t^{2\alpha+1} f(t) j_{\alpha}(2\pi rt) \, dt.
\end{equation}
Letting $\alpha=n/2-1$ in \eqref{hankel}, we recover the Fourier transform of a radial function (cf. \eqref{fourierradial}). The Hankel transform belongs to a larger class of operators, introduced by De Carli in \cite{DC}, namely
$$
	\overline{\mathcal{L}}=\Bigg\{ \overline{\mathcal{L}}^\alpha_{\nu,\mu}f(r) =r^\mu \int_0^\infty (rt)^\nu f(t) J_\alpha(rt)\, dt  \Bigg\}, \quad \alpha\geq -1/2,
$$
with $\mu,\nu\in \R$. In particular, operators from  $\overline{\mathcal{L}}$ appear in the Fourier transform of a radial function multiplied by a spherical harmonic (cf. \cite{DCGT}, \cite{DCG}, \cite{SWfourier}, \cite{Ya}). In the present work, we consider operators from $\overline{\mathcal{L}}$, written in terms of the normalized Bessel function $j_\alpha$, i.e.,
\begin{equation}
	\label{generaloperator}
	\mathcal{L}^\alpha_{\nu,\mu}f(r) =r^\mu \int_0^\infty (rt)^\nu f(t) j_\alpha(rt)\, dt.
\end{equation}
Notice that \eqref{besselsimple} allows us to rewrite \eqref{generaloperator} in terms of an operator from $\overline{\mathcal{L}}$ as
$$
	\mathcal{L}^\alpha_{\nu,\mu}f= 2^\alpha \Gamma (\alpha+1)\overline{\mathcal{L}}^\alpha_{\nu-\alpha,\mu}f. 
$$

We list the following examples of classical transforms written in terms of \eqref{generaloperator}. Here we denote by $F$ a radial function of $n$ variables, and $F(x)=f_0(|x|)$.
\begin{enumerate}
	\item Since $j_{-1/2}(z)=\cos z$, the Fourier cosine transform, denoted by $\hat{f}_{\cos}$, corresponds to the transform $\mathcal{L}^{-1/2}_{0,0}f$.
	\item  The Fourier transform of a radial function for $n\geq 2$ (see \eqref{fourierradial})  satisfies  
	$$
		\hat{F}(y)=|\S^{n-1}|\mathcal{L}^{n/2-1}_{n-1,-(n-1)}f_0(2\pi|y|).
	$$
	\item The classical Hankel transform $H_\alpha$ (see \eqref{hankel}) can be written as
	$$
		H_\alpha f(y)=\frac{2\pi^{\alpha+1}}{ \Gamma(\alpha+1)}  \mathcal{L}^\alpha_{2\alpha+1,-(2\alpha+1)}f(2\pi y).
	$$
	\item If $n\geq 2$ and $\psi_k$ is a solid spherical harmonic of degree $k$, then
	$$
		\widehat{\psi_k F}(y)= \psi_k(y)\cdot  2\pi^{n/2} \bigg(\frac{\pi}{i}\bigg)^k \mathcal{L}^\alpha_{2\alpha+1,-(2\alpha+1)}f_0(2\pi|y|),
	$$
	with $\alpha= (n+2k-2)/2$ (see \cite[Ch. IV,~Theorem 3.10]{SWfourier}).
	\item Let $\mathcal{F}_k$ denote the Dunkl transform, defined by means of a root system $R\subset \R^n$, a reflection group $G\subset O(n)$,  and a multiplicity function $k:R\to \R$ that is $G$-invariant. If $f$ is a radial function defined on $\R^n$, then 
	$$
		\mathcal{F}_kf = H_{n/2-1+\langle k\rangle}f,
	$$
	where $\langle k \rangle = \frac{1}{2} \sum_{x\in R} k(x)$ (cf. \cite{Du, Ro} and the references therein). We also refer the reader to \cite{BKO}, where a generalization of the Dunkl transform is introduced, and \cite{GIT}, where uncertainty principle relations are obtained for this new transform.
	\item Since $j_{1/2}(z)=\sin z/z$, the Fourier sine transform (denoted by $\hat{f}_{\sin}$) equals $\mathcal{L}^{1/2}_{1,0}f$.
\end{enumerate}
				
The goal of this paper is to obtain necessary and/or sufficient conditions on $f$ for \eqref{generaloperator} to converge uniformly on $[0,\infty)$, under the restriction $0\leq \mu+\nu\leq \alpha+3/2$. Outside this range, we give sufficient conditions for the pointwise convergence of \eqref{generaloperator} and the corresponding ones concerning uniform convergence on subsets of $\R_+$. By the \emph{uniform convergence} of $\mathcal{L}^\alpha_{\nu,\mu}f$, we mean that the sequence of partial integrals
$$
	r^{\mu}\int_{0}^{N} (rt)^\nu f(t) j_\alpha(rt)\, dt, \quad N\in \N,
$$
converges uniformly. Equivalently, $\mathcal{L}^\alpha_{\nu,\mu}f$ converges uniformly if and only if
\begin{equation}\label{remainder}
	r^{\mu}\int_{M}^{N} (rt)^\nu f(t) j_\alpha(rt)\, dt \to 0, \quad \text{as }N>M\to \infty,
\end{equation}
uniformly in $r$. We refer to the integral in \eqref{remainder} as the \emph{Cauchy remainder}.

Let us first make an observation showing a key difference between a general operator \eqref{generaloperator} and the sine and cosine transform. From the fact that
\begin{equation*}
	\label{abstrans}
	|\mathcal{L}_{\nu,\mu}^\alpha f(r)|\lesssim r^{\mu+\nu} \int_0^1 t^\nu|f(t)|\, dt + r^{\mu+\nu-\alpha-1/2}\int_1^\infty t^{\nu-\alpha-1/2}|f(t)|\, dt
\end{equation*}
(see \eqref{est} in Section~\ref{sec2}) we have that the absolute convergence of $\mathcal{L}_{\nu,\mu}^\alpha f$ follows from the conditions 
\begin{equation}
	\label{integrability}
	t^\nu f(t)\in L^1(0,1),\quad t^{\nu-\alpha-1/2}f(t)\in L^1(1,\infty).
\end{equation} 
However, unlike the case of $\hat{f}_{\sin}$ or $\hat{f}_{\cos}$, the uniform convergence of $\mathcal{L}_{\nu,\mu}^\alpha f$ does not necessarily follow from the integrability conditions \eqref{integrability} that imply its absolute convergence. This is because the kernel
\begin{equation}
	\label{generalkernel}
	K_{\nu,\mu}^\alpha(t,r) = r^{\mu} (rt)^\nu j_\alpha(rt)
\end{equation}
of $\mathcal{L}^\alpha_{\nu,\mu}$ need not be uniformly bounded. Indeed, if we consider the choice of parameters $\mu=-1$, $0<\nu<\alpha+1/2$ and $f(t)=t^{\mu}=1/t$, the conditions in \eqref{integrability} hold, but the Cauchy remainder	
$$
	r^{\mu}\int_{1/(2r)}^{1/r} (rt)^\nu f(t) j_{\alpha} (rt)\, dt \asymp \frac{1}{r} \int_{1/(2r)}^{1/r} t^{-1}\, dt = \frac{\log 2}{r}
$$
does not vanish  as $r\to 0$ (cf. \eqref{smallz} below).

However, there is a special case when the uniform convergence of $\mathcal{L}_{\nu,\mu}^\alpha f$ follows from \eqref{integrability} (cf. Proposition~\ref{unifrough}), namely when 
\begin{equation}\label{paramidentity}
	\mu+\nu=\alpha+1/2.
\end{equation} In particular, the operators representing $\hat{f}_{\sin}$ ($\alpha=1/2,\nu=1,\mu=0$) and $\hat{f}_{\cos}$ ($\alpha=-1/2, \nu=\mu=0$) satisfy \eqref{paramidentity}.

The two main results of the present paper are the following: 

\begin{theorem}\label{thm1}
	Let $\nu\in \R$ and $\mu=-\nu$.  Let $f:\R_+\to \C$ be such that $t^{\nu} f(t)\in  L^1(0,1)$, and
	\begin{align}
		|f(M)| &= o\big( M^{-\nu-1}\big) & \text{ as }M\to \infty,  \label{thm<}\\
		\int_M^\infty t^{\nu-\alpha-1/2}|df(t)| &= o\big(M^{-\alpha-3/2}\big) & \text{ as } M\to \infty. \label{thm>}		
	\end{align}
	Then, a necessary and sufficient condition for $\mathcal{L}^\alpha_{\nu,\mu}f(r)$ to converge uniformly on $\R_+$ is that
	\begin{equation}\label{necsuf}
		\bigg| \int_0^\infty t^{\nu} f(t)\, dt\bigg|<\infty.
	\end{equation}
\end{theorem}

\begin{theorem} \label{sinethm}
	Let $\nu,\mu\in \R$ be such that $0<\mu+\nu\leq \alpha+3/2$. Let $f:\R_+\to \C$ be such that $t^{\nu} f(t)\in  L^1(0,1)$. If the conditions
	\begin{align}
		|f(M)| &= o\big(M^{\mu-1}\big) & \quad \text{as }M\to \infty, \label{cond1}\\
		\int_M^\infty t^{\nu-\alpha-1/2}|df(t)| 
		&= o\big(M^{\mu+\nu-\alpha-3/2}\big)  & \quad \text{as }M\to\infty. \label{cond2}
	\end{align}
	are satisfied, then $\mathcal{L}_{\nu,\mu}^\alpha f$ converges uniformly on $\R_+$.
\end{theorem}
Observe that conditions \eqref{thm<} and \eqref{thm>} are the same as \eqref{cond1} and \eqref{cond2}, respectively, for the particular case $\mu=-\nu$.
			
The theorems above generalize the following results obtained by Dyachenko, Liflyand and Tikhonov (\cite{DLT}).
\begin{thmletter}\label{cosuc}
	Let $f\in L^1(0,1)$ be vanishing at infinity and such that
	\begin{equation}
		\label{vanishingvariation}
		\int_M^\infty  |df(t)| =o\big(M^{-1}\big) \quad \text{as }M\to \infty.
	\end{equation}
	Then, $$\int_0^\infty f(t)\cos rt\, dt$$
	converges uniformly if and only if $\int_0^\infty f(t)\, dt$ converges.
\end{thmletter}
\begin{thmletter}\label{sinuc}
	Let $f$ be vanishing at infinity and such that $tf(t)\in L^1(0,1)$, and assume \eqref{vanishingvariation} holds. Then, $$\int_0^\infty f(t)\sin rt\, dt$$ converges uniformly.
\end{thmletter}
Note that Theorems~\ref{cosuc} and \ref{sinuc} are particular cases of Theorems~\ref{thm1} and \ref{sinethm}, whenever $\alpha=-1/2$, $\nu=\mu=0$ and $\alpha=1/2$, $\nu=1$, $\mu=0$, respectively. Note that if $f$ vanishes at infinity, \eqref{thm>} and \eqref{cond2} imply \eqref{thm<} and \eqref{cond1}, respectively. In fact, for functions vanishing at infinity, conditions \eqref{thm<} and \eqref{cond1} may be redundant for certain parameters, thus we present alternative statements to those of Theorems~\ref{thm1}~and~\ref{sinethm} (namely, Theorems~\ref{costhmvanish}~and~\ref{sinethmvanish}, respectively).

In view of the respective relationship of Theorems~\ref{thm1}~and~\ref{sinethm} with Theorems~\ref{cosuc}~and~\ref{sinuc}, we will call  $\mathcal{L}^{\alpha}_{\nu,\mu}f$ with $\nu=-\mu$ (or simply $\mathcal{L}^{\alpha}_{\nu,-\nu}f$) \emph{cosine-type} transforms, and  $\mathcal{L}_{\nu,\mu}^{\alpha}f$ with $0<\mu+\nu\leq \alpha+3/2$ \emph{sine-type} transforms. 

We present a picture showing the range of the parameters $\mu$ and $\nu$ for which $\mathcal{L}^\alpha_{\nu,\mu}$ is a sine or cosine-type transform, given a fixed $\alpha>-1/2$. 

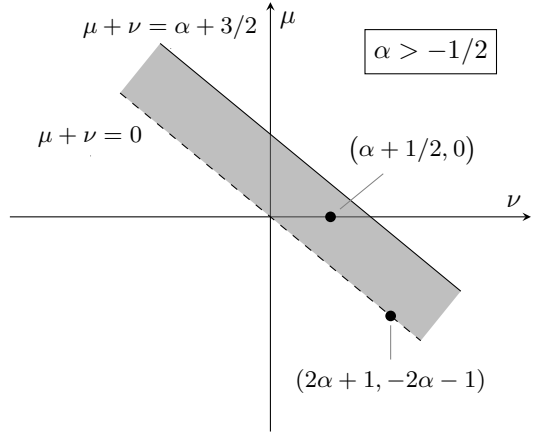
\begin{center}
	\begin{tikzpicture}
	\begin{axis}[axis lines=middle, 
	xlabel={$\nu$}, 
	ylabel={$\mu$}, 
	ymin=-6.5,
	ymax=6.5,
	xmin=-6.5,
	xmax=6.5,
	ticks=none]
	\draw[name path=1, dashed] (-3.75,3.75) -- (3.75,-3.75);
	\addplot[name path=2, color=black, domain=-2.75:4.75] { -x+2.5};
	\addplot[gray, fill opacity=0.5] fill between[of=1 and 2];
	\fill (3,-3) circle[radius=2pt];
	\fill (1.5,0) circle[radius=2pt];
	\node at (axis cs:3,-3) [pin=-90:{\small $(2\alpha+1,-2\alpha-1)$}] {};
	\node at (axis cs: 1.5,0) [pin=80:{\small $\big(\alpha+1/2,0\big)$}] {};
	\node at (axis cs:-4.5,1.75) [pin={[pin distance=0cm] {\small $\mu+\nu=0$}},inner sep=1pt] {};
	\node at (axis cs:-2.5,5) [pin={[pin distance=0cm] {\small $\mu+\nu = \alpha+3/2$}},inner sep=1pt] {};
	\node at (axis cs:4,5) [rectangle, draw=black] {$\alpha>-1/2$};
	\end{axis}
	\end{tikzpicture}	
\end{center}
Every point on the dashed line $\mu=-\nu$ corresponds to a cosine-type transform, and the point $(2\alpha+1,-2\alpha-1)$ lying on such line represents the Hankel transform of order $\alpha$. The area between the dashed line $\mu=-\nu$ (not included) and the line $\mu+\nu=\alpha+3/2$ (included) corresponds to the sine-type transforms. 

The extreme case $\alpha=-1/2$ is the only choice of $\alpha$ for which the operator $\mathcal{L}^{\alpha}_{\alpha+1/2,0}$ does not correspond to a sine-type transform, since $\mathcal{L}^{-1/2}_{0,0}f = \hat{f}_{\cos}$.

For every point of the plane outside the grey strip $0\leq \mu+\nu\leq \alpha+3/2$, we  give sufficient conditions on $f$ that guarantee the pointwise convergence of $\mathcal{L}^\alpha_{\nu,\mu}f$, as well as the uniform convergence on certain subintervals of $[0,\infty)$ (see Section~\ref{sectionpointwise}).

Any function $f$ we consider in this work is complex-valued and defined on $\R_+$, unless otherwise specified (here $\R_+:=[0,\infty)$). We also assume $f$ is locally of bounded variation and locally integrable on $(0,\infty)$. By $f\lesssim g$ and $f\gtrsim g$ we mean that there exist positive constants $C,C'$ such that $f\leq Cg$ and $f\geq C'g$, respectively, and we write $f\asymp g$ if $f\lesssim g$ and $f\gtrsim g$ simultaneously.

The paper is organized in the following way. In Section~\ref{sec2} we present the basic concepts that we will use. Subsection~\ref{subsecbessel} is devoted to the Bessel functions; first we list several of their known properties, and we obtain estimates of integrals containing $j_\alpha$. We emphasize that Lemma~\ref{mainlemma} provides the key estimate to be used throughout this work.  In Subsection~\ref{subsecgm} we 
define the class of \emph{general monotone} ($GM$) functions. To give a flavour, we use the $GM$ property to generalize the following results (that follow from Theorems~\ref{cosuc}~and~\ref{sinuc}, respectively; see also \cite{Moricz-Monotone}):
\begin{thmletterprime}
	If $f\in GM$ and $f\in L^1(0,1)$, then 
	$$\hat{f}_{\cos} \quad \text{converges uniformly if and only if}\quad  \bigg|\int_0^\infty f(t)\, dt\bigg|<\infty.
	$$	
\end{thmletterprime}
\begin{thmletterprime}
	If $f\in GM$ and $tf(t)\in L^1(0,1)$, then
	$$
	\hat{f}_{\sin} \quad \text{converges uniformly if and only if}\quad t|f(t)|\to 0 \quad \text{as }\to \infty.
	$$
\end{thmletterprime}
In Section~\ref{sectionpointwise} we obtain sufficient conditions for the pointwise convergence of \eqref{generaloperator} in the whole range of parameters, and for its uniform convergence on subintervals of $\R_+$, using both integrability of the functions and conditions on their variation. In Sections~\ref{secu.c.}~and~\ref{sineu.c.} we study the uniform convergence of cosine-type and sine-type transforms, respectively. The hypotheses used in such sections mainly depend on variation conditions of $f$. We also give the corresponding statements for $GM$ functions. To conclude Section~\ref{sineu.c.}, we give several examples showing the sharpness of the obtained results, and compare the sufficient conditions obtained in Section~\ref{sectionpointwise} (namely Corollary~\ref{roughcorol}) with those of Theorems~\ref{sinethm}~and~\ref{sinethmvanish}. 

	\section{Preliminary concepts}\label{sec2}
	\subsection{Bessel functions}\label{subsecbessel}
	\textbf{Basic properties}.	Here we list several properties of the normalized Bessel function $j_\alpha(z)$, which can be found in \cite[Chapter VII]{EMOT}. In what follows we will assume $z\in \R_+$. We start with the representation by power series:
	\begin{equation*}\label{bessel}
	j_\alpha(z) = \Gamma(\alpha+1)\sum_{n=0}^\infty \frac{(-1)^n (z/2)^{2n}}{n!\Gamma(n+\alpha+1)}.
	\end{equation*}
	Such series converges uniformly and absolutely on any bounded interval. In particular, for $z\leq 1$,
	\begin{equation}\label{smallest}
	|1-j_\alpha(z)| \leq Cz^2,
	\end{equation}
	with $C< 1$, and therefore
	\begin{equation}
	\label{smallz}
	j_{\alpha}(z)\asymp 1.
	\end{equation}
	Moreover, we have the following asymptotic estimate (cf. \cite{SWfourier}):
	\begin{equation}\label{expansionatinfty}
	j_\alpha(z) = \frac{C_\alpha}{z^{\alpha+1/2}}\cos\bigg(z-\frac{\pi(\alpha+1/2)}{2}\bigg) + O\big(z^{-\alpha-3/2}\big), \quad z\to \infty,
	\end{equation}
	and since $|j_\alpha(z)|\leq j_\alpha(0)=1$ for all $z>0$, then
	\begin{equation}\label{est}
	|j_\alpha(z)|\lesssim \min\bigg\{1,\frac{1}{z^{\alpha+1/2}}\bigg\} \quad \text{for all }z>0.
	\end{equation}	
	Finally, we have the following property concerning the derivatives of $j_\alpha$:
	\begin{equation}\label{dj}
	\frac{d}{dz}\Big( z^{2\alpha+2}j_{\alpha+1}(z)\Big) = (2\alpha+2)z^{2\alpha+1}j_\alpha(z), \quad \alpha\geq -1/2,
	\end{equation}
	from which we deduce
	\begin{equation}\label{dj2}
	\frac{d}{dz} j_{\alpha+1}(z) = \frac{2\alpha+2}{z}\big(j_\alpha(z) -j_{\alpha+1}(z)\big), \quad \alpha\geq -1/2.
	\end{equation}\newline
	
	\noindent \textbf{Auxiliary lemmas}. We will also need upper estimates for the primitive function of $t^\nu j_\alpha(rt)$. We start by rewriting $\int_M^N t^\nu j_\alpha(rt)\, dt$ in terms of higher order Bessel functions. 
	\begin{lemma}\label{primlemma1}
		Let $\alpha\geq -1/2$, $r>0$ and $0<M<N$. Then, for any $k\geq 1$  and $\nu\in \R$ such that $\nu \neq 2(\alpha+\ell)+1$ with $\ell=0,\ldots ,k-1$, 
		\begin{align}
		\int_M^N t^\nu j_\alpha(rt)\, dt & = \sum_{i=1}^kC_{i,\nu,\alpha}\big( N^{\nu+1}j_{\alpha+i}(rN)-M^{\nu+1}j_{\alpha+i}(rM)\big) \nonumber \\
		&\phantom{=}  + C'_{k,\nu,\alpha}\int_M^N t^\nu j_{\alpha+k}(rt)\, dt,\label{prim}
		\end{align}
		where the constants $C_{i,\nu,\alpha}$, $C'_{k,\nu,\alpha}$ are nonzero.
	\end{lemma}
	\begin{proof}
		We prove this statement by induction on $k$. For $k=1$, we can rewrite the integral on the left hand side of \eqref{prim} as $\int_M^N t^{\nu-2\alpha-1} t^{2\alpha+1}j_\alpha(rt)\, dt$,
		and the result follows after integrating by parts together with \eqref{dj}. In this case we have $C_{1,\nu,\alpha} = \frac{1}{2\alpha+2}$ and $C'_{1,\nu,\alpha}=-\frac{\nu-2\alpha-1}{2\alpha+2}$.
		
		If \eqref{prim} holds for some $k\geq 1$, since
		$$
		C'_{k,\nu,\alpha}\int_M^N t^\nu j_{\alpha+k}(rt)\, dt= C'_{k,\nu,\alpha}\int_M^N t^{\nu-2(\alpha+k)-1} t^{2(\alpha+k)+1} j_{\alpha+k}(rt)\, dt,
		$$
		the result follows similarly as before, where in this case we obtain $C_{k+1,\nu,\alpha}= \frac{C'_{k,\nu,\alpha}}{2(\alpha+k)+2}$ and $C'_{k+1,\nu,\alpha} = -C'_{k,\nu,\alpha}\frac{\nu -2(\alpha+k)-1}{2(\alpha+k)+2}$.
	\end{proof}
	\begin{lemma} \label{primlemma2}
	Under the assumptions of Lemma~\ref{primlemma1}, we have, for any $\nu \in \R$ such that $\nu =2(\alpha+\ell)+1$ with some $\ell \in \N\cup\{0\}$,
	$$
	\int_M^N t^\nu j_\alpha(rt)\, dt = \sum_{i=1}^{\ell+1}C_{i,\nu,\alpha}\big( N^{\nu+1}j_{\alpha+i}(rN)-M^{\nu+1}j_{\alpha+i}(rM)\big),
	$$
	where all the constants $C_{i,\nu,\alpha}$ coincide with those of Lemma~\ref{primlemma1}. 
	\end{lemma}
	\begin{proof}
	If $\ell=0$, the result immediately follows from \eqref{dj}. If $\ell>0$, we can apply Lemma~\ref{primlemma1} with $\nu' = 2(\alpha+\ell-1)+1$ in place of $\nu$, and then by \eqref{dj},
	\begin{align*}
	C'_{\ell,\nu,\alpha}\int_M^N t^{2(\alpha+\ell)+1}j_{\alpha+\ell}(rt)\, dt  =  C_{\ell+1,\nu,\alpha}\big( N^{\nu+1}j_{\alpha+\ell+1}(rN)  -M^{\nu+1} j_{\alpha+\ell+1}(rM)\big),
	\end{align*}
	where $C_{\ell+1,\nu,\alpha} = \frac{C'_{\ell,\nu,\alpha}}{2(\alpha+\ell)+1}$.
	\end{proof}
	\begin{remark}\label{rem=0}
	We can allow $M=0$ in Lemmas~\ref{primlemma1}~and~\ref{primlemma2}  whenever $\nu> -1$.
	\end{remark}
	\begin{lemma}\label{primlemma3}
		Let $\alpha\geq -1/2$, $r>0$ and $0<M<N$. For any  $\nu\in \R$ and any $k\geq 1$ such that $ \nu \neq \alpha+k-1/2$, we have
		\begin{equation}\label{primitiveestimate}
		\bigg| \int_M^N t^\nu j_\alpha(rt)\, dt\bigg|\lesssim \frac{1}{r^{\alpha+1/2}} \sum_{i=1}^k \frac{1}{r^i}\big(N^{\nu-i-\alpha+1/2}+M^{\nu-i-\alpha+1/2}\big).
		\end{equation}
	\end{lemma}
	\begin{proof}
		If $\nu$ is as in Lemma~\ref{primlemma2}, the estimate follows by just applying \eqref{est}. On the contrary, if $\nu$ is as in Lemma~\ref{primlemma1}, we estimate the sum of \eqref{prim} in a similar way, whilst since $\nu -\alpha-k-1/2 \neq -1$,
		\begin{align*}
		\bigg|\int_M^N t^\nu j_{\alpha+k}(rt)\, dt \bigg|&  \lesssim \frac{1}{r^{\alpha+k+1/2}}\int_M^N t^{\nu-\alpha-k-1/2}\, dt\\
		&\lesssim \frac{1}{r^{k+\alpha+1/2}}\big(N^{\nu-k-\alpha+1/2}+M^{\nu-k- \alpha+1/2}\big),
		\end{align*}
		which coincides precisely with the $k$-th term of the sum in \eqref{primitiveestimate}.
	\end{proof}
	Since the Bessel function $j_\alpha(z)$ is continuous, if we denote by $g^\nu_{\alpha,r}(t)$ the primitive function of $t^\nu j_\alpha(rt)$, we have, in virtue of the fundamental theorem of calculus,
		\begin{equation}\label{primformula}
		g^\nu_{\alpha,r}(t) = \begin{dcases}	 \int_0^t s^\nu j_\alpha(rs)\, ds, & \text{if } \nu \geq\alpha+1/2 \text{ and }\alpha>-1/2,\text{ or if }\nu>\alpha+1/2,\\ 	
	-\int_t^\infty s^\nu j_\alpha(rs)\, ds, &\text{if }\nu < \alpha+1/2, \\
		\dfrac{\sin rt}{r}, & \text{if }\nu=0 \text{ and }\alpha=-1/2.
		\end{dcases}
		\end{equation}
		\begin{remark}
			Note that 
		$$
		\int_0^x t^\nu j_\alpha(rt)\, dt = \frac{x^{\nu+1}}{\nu+1} {}_1F_2\bigg( \frac{1}{2}(\nu+1) ; \frac{1}{2}(\nu+3), \alpha+1 ; -\frac{(rx)^2}{4}\bigg), \quad \nu>-1,
		$$
		where  ${}_pF_q$ denotes the generalized hypergeometric function (see \cite[Ch. 6]{luke}).
	\end{remark}
	
	We are now in a position to obtain the upper bound of  \eqref{primformula}.
	
	\begin{lemma}\label{mainlemma}
		The estimate
				\begin{equation}
				\label{Primest}
				|g^\nu_{\alpha,r}(t)| \lesssim \frac{t^{\nu-\alpha-1/2}}{r^{\alpha+3/2}}, \quad \alpha\geq -1/2, \quad \nu\in \R,
				\end{equation}
holds.
	\end{lemma}
	\begin{proof}
		We distinguish two cases: $\nu\neq \alpha+1/2$, or $\nu=\alpha+1/2$. In the first case, estimate \eqref{Primest} follows readily applying Lemma~\ref{primlemma3} with $k=1$ and letting $M\to 0$ or $N \to \infty$ if $\nu>\alpha+1/2$ or $\nu< \alpha+1/2$, respectively. 
		
		If $\nu=\alpha+1/2$, and $\alpha=-1/2$, \eqref{Primest} follows immediately from \eqref{primformula}. For $\alpha>-1/2$, we can apply Lemma~\ref{primlemma1} with $k=2$ (see also Remark~\ref{rem=0}) to obtain
		
				\begin{align*}
				|g^{\alpha+1/2}_{\alpha,r}(t)|& =\bigg| C_1t^{\alpha+3/2}j_{\alpha+1}(rt) + C_2 t^{\alpha+3/2}j_{\alpha+2}(rt) +C_3 \int_0^t s^{\alpha+1/2} j_{\alpha+2}(rs)\, ds\bigg| \\
				&\lesssim \frac{1}{r^{\alpha+3/2}}+\frac{1}{tr^{\alpha+5/2}} + \int_0^t s^{\alpha+1/2}|j_{\alpha+2}(rs)|\, ds
				\end{align*}
		It follows from \eqref{est} that
		\begin{align*}
		\int_0^t s^{\alpha+1/2}|j_{\alpha+2}(rs)|\, ds \leq \int_{0}^{1/r} s^{\alpha+1/2}\, ds + \frac{1}{r^{\alpha+5/2}} \int_{1/r}^\infty s^{-2}\, ds \lesssim \frac{1}{r^{\alpha+3/2}}.
		\end{align*}
		Collecting the estimates above, we deduce
		\begin{equation*}
		\label{specialest}
		|g^{\alpha+1/2}_{\alpha,r}(t)|\lesssim \frac{1}{r^{\alpha+3/2}}+\frac{1}{tr^{\alpha+5/2}}.
		\end{equation*}
		In particular, it follows from the latter estimate that \eqref{Primest} holds whenever $t\geq 1/r$. 
		
		Finally, if $t<1/r$, using \eqref{primformula} together with \eqref{smallz} we obtain
		$$
		|g^\nu_{\alpha,r}(t)|  \asymp \int_0^t s^{\alpha+1/2}\, ds \asymp t^{\alpha+3/2}<\frac{1}{r^{\alpha+3/2}},
 		$$
 		and the proof is complete.
	\end{proof}
	
	
	\subsection{General Monotonicity} \label{subsecgm}
	
	It is often useful to consider quantitative characteristics of functions that are locally of bounded variation, such as the so-called \emph{general monotonicity} (cf. \cite{GLT}, \cite{LTnach}, \cite{LTparis}, \cite{LTtrends} and \cite{TikGM}). 
	
	\begin{definition}
		Let $\beta:\R_+\to \R_+$. We say that a function $f$ is $\beta$-\emph{general monotone}, written $f\in GM(\beta)$, if there exists $C>0$ such that for every $x>0$,
		$$
		\int_x^{2x}|df(t)|\leq C\beta(x).
		$$
	\end{definition}
	In many cases important $GM(\beta)$ classes are those where $\beta$ depends on the function $f$ itself, rather than on its variation.	We restrict ourselves to the concrete choice of $\beta$ introduced in \cite{LTparis}.
	\begin{definition}\label{defgm}
		We say that $f$ is a $GM$ function, written $f\in GM$, if there exist $C,\lambda>1$ such that for every $x>0$,
		$$
		\int_x^{2x} |df(t)| \leq \frac{C}{x}\int_{x/\lambda}^{\lambda x}|f(t)|\, dt.
		$$
	\end{definition}
	Note that any monotone function is also a $GM$ function.
	
	We could consider more general $GM(\beta)$ classes, such as the one defined in \cite{DLT}, where
	\begin{equation*}\label{gmb}
	\beta(x)=\beta_0(x)=\frac{1}{x} \sup_{s\geq x/\lambda} \int_s^{2s} |f(t)|\, dt.
	\end{equation*}
	It is known that $GM\subsetneq GM(\beta_0)$. However, the latter class is too wide, and may even give no useful information about the variation of $f$, if $\int_x^{2x}|f(t)|\, dt$ is not bounded at infinity. 
	
	The following holds for any $GM(\beta)$ function (see \cite[Lemma~5.2]{LTnach}): 
	\begin{lemma}\label{gmlemma}
		If $f\in GM(\beta)$ and $x>0$, then
		$$
		|f(t)| \leq C\beta(x) + \int_t^{2t} \frac{|f(s)|}{s}\, ds \quad \text{for any }t\in[x,2x].
		$$
	\end{lemma}
	It follows from Lemma~\ref{gmlemma} that if $f\in GM$ and $\lambda\geq 2$ (which can be assumed without loss of generality), one has
	\begin{equation}\label{gmest}
	|f(x)| \leq C \int_{x/\lambda}^{\lambda x} \frac{|f(t)|}{t}\, dt \asymp \frac{1}{x}\int_{x/\lambda}^{\lambda x} |f(t)|\, dt, \quad x>0.
	\end{equation}
	
	Note that the following estimate holds for all $f\in GM$:	
	\begin{align}
	\int_1^\infty t^{\nu-\alpha-1/2}\,|df(t)|&\lesssim \int_{1/2}^{\infty} \frac{1}{t} \int_{t}^{2t}s^{\nu-\alpha-1/2}\, |df(s)|\, dt \asymp \int_{1/2}^{\infty} t^{\nu-\alpha-3/2}\int_{t}^{2t} |df(s)| \, dt \nonumber \\ 
	&\lesssim \int_{1/2}^{\infty} t^{\nu-\alpha-5/2} \int_{t/\lambda}^{\lambda t}|f(s)|\, ds\, dt \lesssim \int_{1/(2\lambda)}^\infty  s^{\nu-\alpha-3/2}|f(s)|\, ds. \label{gmvariation}
	\end{align}
	We can apply the latter inequality in order to replace the hypotheses on the variation of $f$ by integrability conditions (for instance, compare Proposition~\ref{pointwise} with Corollary~\ref{corolpoint} below).

	\section{Pointwise and uniform convergence of $\mathcal{L}^{\alpha}_{\nu,\mu}f$: first approach}\label{sectionpointwise}
	In this section we are interested in finding sufficient conditions on $f$ that guarantee the pointwise convergence of \eqref{generaloperator}. We will see that these sufficient conditions also imply the uniform convergence of \eqref{generaloperator} on certain subintervals of $\R_+$. If we assume that $t^{\nu}f(t)\in L^1(0,1)$, the convergence of $\mathcal{L}_{\nu,\mu}^\alpha f$ at $r_0\in \R_+$ means that
	$$
	\lim_{M\to \infty} \bigg| r_0^{\mu+\nu} \int_0^M t^\nu f(t)j_\alpha(r_0t)\, dt \bigg|<\infty.
	$$
	In contrast with the  criteria for uniform convergence (see Theorems~\ref{thm1}~and~\ref{sinethm}), we do not impose restrictions on the parameters for now. The criterion for convergence at the origin is rather simple:
	\begin{enumerate}[label=(\roman{*})]
		\item If $\mu+\nu<0$, then $\mathcal{L}^\alpha_{\nu,\mu}f(0)$ is not defined.
		\item If $\mu+\nu=0$, the convergence of $\mathcal{L}^\alpha_{\nu,\mu}f(0)$ is equivalent to $\big| \int_0^\infty t^{\nu} f(t)\, dt\big|<\infty$.
		\item If $\mu+\nu>0$, then $\mathcal{L}^\alpha_{\nu,\mu}f(0)=0$.
	\end{enumerate}

Now we study the pointwise convergence of $\mathcal{L}_{\nu,\mu}^\alpha f(r)$ for $r>0$. When possible, we also give sufficient conditions for the uniform convergence on subintervals of $\R_+$. The statements in this section can be subdivided into two categories, depending on their hypotheses. First, we have those relying on integrability of $f$, and secondly, those involving conditions on the variation of $f$.

\subsection{Integrability conditions}
We begin with the statements involving integrability conditions of $f$. \begin{proposition}\label{unifrough}
	Let $f$ be such that $t^\nu f(t)\in L^1(0,1)$ and $t^{\nu-\alpha-1/2}f(t)\in L^1(1,\infty)$. Then $\mathcal{L}^\alpha_{\nu,\mu} f$ converges for $r>0$. Moreover,
	\begin{enumerate}
		\item If $\mu+\nu-\alpha-1/2<0$, then $\mathcal{L}^\alpha_{\nu,\mu} f$ converges uniformly on any interval $[\varepsilon,\infty)$ with $\varepsilon >0$.
		\item If $\mu+\nu-\alpha-1/2>0$, then $\mathcal{L}^\alpha_{\nu,\mu} f$ converges uniformly on any interval $[0,\varepsilon]$ with $\varepsilon >0$.
		\item If $\mu+\nu-\alpha-1/2=0$, then $\mathcal{L}^\alpha_{\nu,\mu} f$ converges uniformly on $\R_+$.
	\end{enumerate}
\end{proposition}
\begin{proof}
	It is clear that the pointwise convergence of $\mathcal{L}^\alpha_{\nu,\mu}$ at $r>0$ is equivalent to $$\int_M^\infty t^\nu f(t)j_\alpha(rt)\, dt =o(1)  \quad \text{as }M\to \infty,$$
	which holds by simply applying \eqref{est} and the fact that $t^{\nu-\alpha-1/2}\in L^1(1,\infty)$. 
	
	Let us now prove the statement concerning uniform convergence. For each of the three cases, since $t^{\nu-\alpha-1/2}f(t)\in L^1(1,\infty)$, it follows from \eqref{est} that  
	$$
	r^{\mu+\nu}\int_M^\infty t^\nu f(t) j_\alpha(rt)\, dt \leq \varepsilon^{\mu+\nu-\alpha-1/2} \int_M^\infty t^{\nu-\alpha-1/2} |f(t)|\, dt = o(1) \quad \text{as }M\to \infty,
	$$
	 that is, the Cauchy remainder vanishes uniformly in $r$ as $M\to \infty$ (in each corresponding interval).
\end{proof}
Proposition~\ref{unifrough} allows us to derive sufficient conditions for the uniform convergence of $\mathcal{L}^{\alpha}_{\mu,\nu}f$ on $\R_+$ whenever $0\leq \mu+\nu\leq \alpha+1/2$.
\begin{corollary}\label{PROPuclacunary}
	Let $0\leq \mu+\nu\leq \alpha+1/2$. If $t^\nu f(t)\in L^1(\R_+)$, then $\mathcal{L}^\alpha_{\nu,\mu} f$ converges uniformly.
\end{corollary}
\begin{proof}
	First, if $0\leq\mu+\nu<\alpha+1/2$, note that since $\alpha\geq -1/2$, $t^\nu f(t)\in L^1(\R_+)$ implies $t^{\nu-\alpha-1/2}f(t)\in L^1(1,\infty)$, so we can apply Proposition~\ref{unifrough} to deduce that $\mathcal{L}^\alpha_{\nu,\mu}f$ converges uniformly on any interval $[\varepsilon,\infty)$ with $\varepsilon>0$, whilst the uniform convergence on the interval $[0,\varepsilon]$ follows from
	$$
	r^{\mu+\nu}\int_M^\infty t^\nu f(t)j_{\alpha}(rt)\, dt\leq \varepsilon^{\mu+\nu} \int_M^\infty t^{\nu}|f(t)|\, dt \to 0 \qquad \text{as }M\to \infty.
	$$
	Secondly, if $\mu+\nu=\alpha+1/2$, then $t^{\nu-\alpha-1/2}f(t)=t^{-\mu}f(t)$, and therefore $t^{\nu}f(t)\in L^1(\R_+)$ implies $t^{-\mu}f(t)\in L^1(1,\infty)$ (since $\nu\geq-\mu$ for every $\alpha\geq -1/2$), and the result follows by Proposition~\ref{unifrough}.
\end{proof}

\subsection{Variational conditions}
The statements of this subsection involve conditions on the variation of $f$. In the case of $GM$ functions, these follow from integrability conditions of $f$ (cf. \eqref{gmest}), allowing us to rewrite certain statements. When possible, we also give sufficient conditions for the uniform convergence of $\mathcal{L}^\alpha_{\nu,\mu}f$ on $\R_+$ that follow after combining the results on the present subsection with those of the previous one.
	\begin{proposition}\label{pointwise}
	Let $f$ be such that $t^\nu f(t)\in L^1(0,1)$ and
	\begin{equation}\label{pointwisecond}
	\int_1^\infty t^{\nu-\alpha-1/2}|df(t)|<\infty \quad \text{and}\quad M^{\nu-\alpha-1/2}|f(M)| \to 0 \quad\text{as }M\to \infty ,
	\end{equation}
	then $\mathcal{L}^\alpha_{\nu,\mu}f(r)$ converges for $r>0$. Moreover, for any $\varepsilon>0$,
	\begin{enumerate}
		\item if $\mu+\nu-\alpha-3/2> 0$, the convergence is uniform on any interval $[0,\varepsilon]$;
		\item if $\mu+\nu-\alpha-3/2< 0$, the convergence is uniform on any interval $[\varepsilon,\infty)$;
		\item if $\mu+\nu-\alpha-3/2=0$, the convergence is uniform on $\R_+$.
	\end{enumerate}
	\end{proposition}
	
		\begin{remark}\label{rmkpointwise}
			\begin{enumerate}[label=(\roman*)]
				\item Note that in the extremal case $\mu+\nu=\alpha+3/2$, the conditions \eqref{pointwisecond} are equivalent to \eqref{cond1} and \eqref{cond2}.
			\item In the case $\nu\geq \alpha+1/2$, if $f$ vanishes at infinity the convergence of $\int_1^\infty t^{\nu-\alpha-1/2}|df(t)|$ implies that $M^{\nu-\alpha-1/2}f(M)\to 0$ as $M\to \infty$. Indeed,
			$$
			M^{\nu-\alpha-1/2}|f(M)| \leq M^{\nu-\alpha-1/2}\int_M^\infty |df(t)|\leq \int_M^\infty t^{\nu-\alpha-1/2}\, |df(t)|,
			$$
			and the right hand side of the latter vanishes as $M\to\infty $. Thus, in this case we only need to assume the convergence of $\int_1^\infty t^{\nu-\alpha-1/2}|df(t)|$ in Proposition~\ref{pointwise}.
				\end{enumerate}
		\end{remark}
	For functions satisfying the $GM$ property, we can derive a version of Proposition~\ref{pointwise} depending on integrability conditions of $f$, which are less restrictive than those from Proposition~\ref{unifrough}. 
	\begin{corollary}\label{corolpoint}
		Let $f\in GM$ be such that $t^\nu f(t)\in L^1(0,1)$. If $t^{\nu-\alpha-3/2}f(t)\in L^1(1,\infty)$, all the statements of Proposition~\ref{pointwise} hold.
	\end{corollary}

	\begin{proof}[Proof of Proposition~\ref{pointwise}]
	We fix $r>0$. Since $t^{\nu}f(t)\in L^1(0,1)$, the convergence of \eqref{generaloperator} is equivalent to
	$$
	\lim_{M\to \infty} \bigg|\int_1^M t^\nu f(t) j_\alpha(rt)\, dt\bigg|<\infty.
	$$
	Note that condition $t^{\nu-\alpha-1/2}|f(t)| \to 0$ as $t\to \infty$ implies that the integrand $t^\nu f(t) j_\alpha(rt)$ vanishes as $t\to \infty$. Integrating by parts, we have
	\begin{align*}
	\int_1^M t^\nu f(t)j_\alpha(rt)\, dt = g^\nu_{\alpha,r}(M)f(M)-g^\nu_{\alpha,r}(1)f(1)- \int_1^M g^\nu_{\alpha,r}(t)\, df(t),
	\end{align*}
	where $g^\nu_{\alpha,r}(t)$ is given by \eqref{primformula}. Now we estimate each term of the latter expression (note that $g_{\alpha,r}^\nu(1) f(1)$ is bounded). It follows from \eqref{Primest} and \eqref{pointwisecond} that
	$$
	|g^\nu_{\alpha,r}(M)f(M)|\lesssim \frac{M^{\nu-\alpha-1/2}}{r^{\alpha+3/2}}|f(M)| \to 0 \quad \text{as }M\to \infty.$$
	Finally, 
	$$
	\int_1^M |g_{\alpha,r}^\nu (t) \, df(t)|\lesssim \frac{1}{r^{\alpha+3/2}}\int_1^M t^{\nu-\alpha-1/2}\, |df(t)| .
	$$
	Thus, the condition $\int_1^\infty t^{\nu-\alpha-1/2}|df(t)|<\infty$ implies that the integral
	$$
	\int_1^\infty |g_{\alpha,r}^\nu (t) \, df(t)|
	$$
	converges, which concludes the part concerning pointwise convergence.
	
	The statement related to uniform convergence is easily proved by simply applying estimates \eqref{est} and \eqref{Primest} to the Cauchy remainder:
	\begin{align*}
	r^{\mu}\bigg|\int_M^N  (rt)^\nu f(t) j_\alpha(rt)\, dt\bigg|&=r^{\mu+\nu} \bigg| g_{\alpha,r}^\nu(N) f(N)- g_{\alpha,r}^\nu(M)f(M)\nonumber -\int_M^N g_{\alpha,r}^\nu(t)\, df(t)\bigg| \nonumber\\ &\lesssim r^{\mu+\nu-\alpha-3/2}\bigg( N^{\nu-\alpha-1/2}|f(N)|+M^{\nu-\alpha-1/2}|f(M)|\nonumber \\ 
	&\phantom{=}+ \int_M^N t^{\nu-\alpha-1/2}\, |df(t)|\bigg).
	\end{align*}
	Thus, the latter expression vanishes 
	\begin{enumerate}
		\item uniformly in $r\in [0,\varepsilon]$ if $\mu+\nu-\alpha-3/2>0$;
		\item uniformly in $r\in [\varepsilon,\infty)$ if $\mu+\nu-\alpha-3/2<0$;
		\item uniformly in $r\in[0,\infty)$ if $\mu+\nu-\alpha-3/2=0$,
	\end{enumerate}
	as $N>M\to \infty$.
 	\end{proof}
 	\begin{proof}[Proof of Corollary~\ref{corolpoint}]
			First of all, note that if $f\in GM$, the condition $t^{\nu-\alpha-3/2}f(t)\in L^1(1,\infty)$ implies that $t^{\nu-\alpha-1/2}f(t)$ vanishes at infinity (see \eqref{gmest}). Furthermore, by \eqref{gmvariation}, we have that all hypotheses of Proposition~\ref{pointwise} are satisfied, and the result follows. 		
 	\end{proof}

Our last statement of this subsection is just a combination of Propositions~\ref{unifrough}~and~\ref{pointwise}.
\begin{corollary}\label{roughcorol}
	Let $f$ be such that $t^\nu f(t)\in L^1(0,1)$. Assume that $\alpha+1/2\leq \mu+\nu< \alpha+3/2$. If the conditions \eqref{pointwisecond} hold, and if $t^{\nu-\alpha-1/2} f(t) \in L^1(1,\infty)$, then $\mathcal{L}^\alpha_{\nu,\mu} f$ converges uniformly.
\end{corollary}
Note that except for the case $\alpha=-1/2$ and $\mu+\nu=0$, the parameters for which Corollary~\ref{roughcorol} can be applied correspond to sine-type transforms.

\subsection{Examples}

Let us discuss an application of Proposition~\ref{pointwise}, which is closely related to the following classical result \cite[Ch. I, Theorem 2.6]{zygmund} (see also \cite[Ch. I, \textsection 30]{bari}):
{\it Let $\varphi(x)$ be either $\sin x$ or $\cos x$. If $a_n\to 0$ and $\{a_n\}\in BV$, or equivalently,
	$$
	\sum_{n=N}^\infty |a_n-a_{n+1}|=o(1)\quad \text{as }N\to \infty,
	$$
	then $\sum_{n=0}^\infty a_n \varphi(nx)$ converges pointwise  in $x\in(0,2\pi)$, and the convergence is uniform on any interval $[\varepsilon, 2\pi-\varepsilon]$, $\varepsilon>0$.
}

A version of the latter statement for the sine and cosine transforms follows from Proposition~\ref{pointwise} (see item 2 of the latter, and note that for the sine and cosine transforms both conditions $\mu+\nu-\alpha-3/2<0$ and $\nu-\alpha-1/2=0$ hold).
\begin{thmletter}
	Let $f,g:\R_+\to \C$ be vanishing at infinity and such that $f\in L^1(0,1)$ and $tg(t)\in L^1(0,1)$. Assume that $f$ and $g$ are of bounded variation on $[\delta,\infty)$ for some $\delta>0$. Then, $\hat{f}_{\cos}(r)$ and $\hat{g}_{\sin}(r)$ converge for every $r>0$, and the convergence is uniform on every interval $[\varepsilon,\infty)$, with $\varepsilon>0$.
\end{thmletter}

Finally, we give an example showing that we cannot guarantee the uniform convergence of $\mathcal{L}^\alpha_{\nu,\mu}f$ on $\R_+$ outside the range of parameters $0\leq \mu+\nu\leq \alpha+3/2$, whenever $f$ satisfies both conditions from \eqref{pointwisecond}. The case $\mu+\nu<0$ is clear, since in this case $\mathcal{L}^\alpha_{\nu,\mu}f(0)$ is not even defined. The case $\mu+\nu>\alpha+3/2$ is more involved.

Let 
$$
f(t)=\begin{cases}
t^{-\nu},&\text{if }t<2,\\
\dfrac{t^{-\nu+\alpha+1/2}}{ (\log t)^2}, &\text{if }t\geq 2.
\end{cases}
$$
On the one hand, since for any $\nu\in \R$ and $\alpha\geq -1/2$ one has
$$
f'(t)= (-\nu+\alpha+1/2)\frac{t^{-\nu-1/2+\alpha}}{(\log t)^2}-2\frac{t^{-\nu-1/2+\alpha}}{(\log t)^3},
$$
it is clear that
$$
\int_1^\infty t^{\nu-\alpha-1/2}|df(t)|\leq 1+\int_2^\infty t^{\nu-\alpha-1/2}|f'(t)|\, dt \lesssim \int_2^\infty \frac{1}{t(\log t)^2}\, dt<\infty.
$$
On the other hand, for $t\geq 2$
$$
t^{\nu-\alpha-1/2}f(t)=\frac{1}{(\log t)^2}\to 0 \qquad \text{as }t\to \infty,
$$
and hence $f$ satisfies both conditions from \eqref{pointwisecond}. Let us now prove that $\mathcal{L}^\alpha_{\nu,\mu}f$ does not converge uniformly on $\R_+$. Let $2<M<N$. Integration by parts along with property \eqref{dj} of $j_\alpha$ yields the following equality:
\begin{align*}
r^{\mu+\nu}\bigg|\int_M^N t^\nu f(t)j_{\alpha}(rt)\, dt \bigg| & =r^{\mu+\nu} \bigg|\frac{1}{2\alpha+2} \bigg[\frac{t^{\alpha+3/2}}{(\log t)^2}j_{\alpha+1}(rt) \bigg]_M^N \\ 
&\phantom{=} +\frac{\alpha+1/2}{2\alpha+2}\int_M^N  \frac{t^{\alpha+1/2}}{(\log t)^2}j_{\alpha+1}(rt)\, dt \\
&\phantom{=} + \frac{2}{2\alpha+2}\int_M^N \frac{t^{\alpha+1/2}}{(\log t)^3}j_{\alpha+1}(rt)\, dt \bigg|=: a_0+b_0+c_0
\end{align*}

First,
$$
a_0\asymp r^{\mu+\nu} \bigg| \frac{N^{\alpha+3/2}}{(\log N)^2}j_{\alpha+1}(rN)-\frac{M^{\alpha+3/2}}{(\log M)^2}j_{\alpha+1}(rM)\bigg|.
$$
If we choose $r=(\log M)^{2/(\mu+\nu-\alpha-3/2)}$ and $M$ so that $j_{\alpha+1}(rM)\asymp (rM)^{-\alpha-3/2}$ (such $M$ can be found through \eqref{expansionatinfty}), we obtain by letting $N\to \infty$,
$$
a_0\asymp \frac{r^{\mu+\nu-\alpha-3/2}}{(\log M)^2} =1.
$$ 
We now prove that both terms $b_0$ and $c_0$ vanish as $N>M\to \infty$ (for this particular choice of $r$). If we prove such claim, then it follows that $\mathcal{L}^\alpha_{\nu,\mu}f$ does not converge uniformly on $\R_+$. Let us proceed to estimate $b_0$ from above first. Using again integration by parts and \eqref{dj}, we obtain
\begin{align*}
r^{\mu+\nu}\bigg|\int_M^N  \frac{t^{\alpha+1/2}}{(\log t)^2}j_{\alpha+1}(rt)\, dt \bigg|&=r^{\mu+\nu}\bigg| \frac{1}{2\alpha+4}\bigg[ \frac{t^{\alpha+3/2}}{(\log t)^2} j_{\alpha+2}(rt)\bigg]_M^N \\
&\phantom{=}  +\frac{\alpha+5/2}{2\alpha+4} \int_M^N \frac{t^{\alpha+1/2}}{(\log t)^2}j_{\alpha+2}(rt)\, dt \\
&\phantom{=}+\frac{2}{2\alpha+4}\int_M^N \frac{t^{\alpha+1/2}}{(\log t)^3}j_{\alpha+2}(rt)\, dt \bigg|=: a_1+b_1+c_1.
\end{align*}
By \eqref{est}, it is clear that
$$
a_1\lesssim r^{\mu+\nu-\alpha-5/2}\bigg( \frac{1}{M(\log M)^2}+ \frac{1}{N(\log N)^2} \bigg) \lesssim  \frac{(\log M)^{2\big|\frac{\mu+\nu-\alpha-5/2}{\mu+\nu-\alpha-3/2}\big|}}{M} \to 0
$$
as $N>M\to \infty$, as for $b_1,c_1$, we note that
\begin{align*}
b_1+c_1 & \lesssim r^{\mu+\nu}\int_{M}^{N}\frac{t^{\alpha+1/2}}{(\log t)^2}|j_{\alpha+2}(rt)|\, dt \lesssim r^{\mu+\nu-\alpha-5/2}\int_M^N \frac{1}{t^{2}(\log t)^2}\, dt \\
&\leq \frac{r^{\mu+\nu-\alpha-5/2}}{M}\leq \frac{(\log M)^{2\big|\frac{\mu+\nu-\alpha-5/2}{\mu+\nu-\alpha-3/2}\big|}}{M} \to 0
\end{align*}
as $N>M\to \infty$. Let us now inspect the term $c_0$. Once again, integration by parts and \eqref{est} yield
\begin{align*}
r^{\mu+\nu} \bigg|\int_M^N  \frac{t^{\alpha+1/2}}{(\log t)^3}j_{\alpha+1}(rt)\, dt \bigg| \lesssim r^{\mu+\nu}\bigg( \bigg|\bigg[ \frac{t^{\alpha+3/2}}{(\log t)^3} j_{\alpha+2}(rt)\bigg]_M^N\bigg| +\int_M^N \frac{t^{\alpha+1/2}}{(\log t)^3}|j_{\alpha+2}(rt)|\, dt \bigg),
\end{align*}
and it can be shown similarly as above that the latter vanishes as  $N>M\to \infty$.  Therefore, we conclude that $\mathcal{L}^\alpha_{\nu,\mu} f$ does not converge uniformly.

	\section{Uniform convergence of $\mathcal{L}^\alpha_{\nu,\mu}f$ with $\mu+\nu=0$}\label{secu.c.}

	In the present section we investigate necessary and sufficient conditions for the uniform convergence of the transforms $\mathcal{L}^\alpha_{\nu,\mu}f$ with $\mu+\nu=0$ (or equivalently, $\mathcal{L}^\alpha_{\nu,-\nu}f$), as for example, the Hankel transform.
	\subsection{Main Results}
	Additionally to Theorem~\ref{thm1}, we have other uniform convergence criteria for cosine-type transforms that will be stated and proved in this section, namely Theorems~\ref{costhmvanish}~and~\ref{thm1.5}. The former is a direct consequence of Theorem~\ref{thm1} and relies on hypotheses involving the variation of $f$, whilst the latter only depends on the continuity of $f$ and its asymptotic behaviour at infinity.
		
	\begin{remark}
		\begin{enumerate}[label=(\roman*)]
			\item In Theorem~\ref{thm1} we omit the simple case $f\geq 0$, since the uniform convergence of $\mathcal{L}^\alpha_{\nu,-\nu}f$ is clearly equivalent to $t^{\nu}f(t) \in L^1(\R_+)$.
			\item The criterion for the uniform convergence of the Hankel transform can be derived by letting $\nu = 2\alpha+1$ in Theorem~\ref{thm1}, i.e., \textit{if \eqref{thm>} holds, then $H_\alpha f(r)$ converges uniformly if and only if}
			$$
			\bigg|\int_0^\infty t^{2\alpha+1}f(t)\, dt \bigg|<\infty.
			$$		
		\end{enumerate}
	\end{remark}	

In Theorem~\ref{thm1} we do not require that $f$ vanishes at infinity. For functions satisfying the latter property, we have the following simplified statement.

	\begin{theorem}
		\label{costhmvanish}
		Let $\nu\in \R$ and $\mu=-\nu$.  Let $f$ be vanishing at infinity and such that $t^\nu f(t)\in L^1(0,1)$.  Assume that
		\begin{align}
		\int_M^\infty |df(t)| &= o\big( M^{-\nu-1}\big) \quad \text{as }M\to \infty, & \text{if }\nu<\alpha+1/2 \text{ and }\nu>-1, \label{costhmsmall}\\
		\int_M^\infty t^{\nu-\alpha-1/2}|df(t)| &=o\big(M^{-\alpha-3/2}\big) \quad \text{as }M\to \infty, &\text{if }\nu\geq \alpha+1/2\text{ or }\nu\leq -1.\label{costhmbig}
		\end{align}
			Then, condition \eqref{necsuf} is necessary and sufficient  to guarantee the uniform convergence of $\mathcal{L}^\alpha_{\nu,\mu}f(r)$ on $\R_+$.
	\end{theorem}

	We can give an alternative statement to Theorem~\ref{thm1} for $GM$ functions. 
	\begin{corollary}\label{CORcosgm}
		Let $f\in GM$ be real-valued and such that $tf(t)\in L^1(0,1)$. Then $\mathcal{L}^\alpha_{\nu,-\nu}f$ converges uniformly if and only if \eqref{necsuf} holds.
	\end{corollary}
	\begin{remark}
		From the proof of the latter it is clear that the same conclusion holds for complex-valued $f\in GM$ if we also assume \eqref{thm<}.
	\end{remark}
	 	
		As mentioned above, we now prove a different criterion that depends on the continuity of $f$ and its behaviour at infinity. Recall that if
		$$
		F_\nu(x) = -\int_x^\infty t^{\nu}f(t)\, dt,
		$$
and  \eqref{necsuf} holds, the continuity of $f$ implies that $F_\nu'(x) = x^{\nu}f(x)$, in virtue of the fundamental theorem of Calculus.
		\begin{theorem}\label{thm1.5}
			Let $f\in C(1,\infty)$ be such that $t^\nu f(t)\in L^1(0,1)$. Assume that $\alpha>1/2$, and that \eqref{thm<} holds.  Then, the transform $\mathcal{L}^\alpha_{\nu,-\nu}f$
			converges uniformly if and only if \eqref{necsuf} is satisfied.
		\end{theorem}
	Note that the range of $\alpha$ for which Theorem~\ref{thm1.5} is valid is reduced compared to the one of Theorem~\ref{thm1}. We also stress that contrarily to Theorems~\ref{thm1}~and~\ref{costhmvanish}, Theorem~\ref{thm1.5} does not require any control on the variation of $f$.
		\begin{remark}
			Whenever $f$ vanishes at infinity and $\nu>-1$, if $\nu<\alpha+1/2$ then \eqref{costhmsmall} implies \eqref{thm<}, and if $\nu\geq \alpha+1/2$, \eqref{costhmbig} implies \eqref{thm<}. However, the converse is not true. Indeed, consider $f(t)=t^{-\nu-2}\sin t$, with $t>1$. It is clear that \eqref{thm<} holds, and thus $\mathcal{L}^\alpha_{\nu,\mu}f$ converges uniformly, but since $f'(t) = -(\nu+2)t^{-\nu-3}\sin t+t^{-\nu-2}\cos t$, one has
			\begin{align*}
			M^{\nu+1}\int_M^\infty |f'(t)|\, dt  &\asymp M^{\nu+1}\int_M^\infty t^{-\nu-2}\, dt \asymp 1, \\
			M^{\alpha+3/2}\int_M^\infty t^{\nu-\alpha-1/2}|f'(t)|\, dt &\asymp M^{\alpha+3/2}\int_M^\infty t^{-\alpha-5/2}\, dt \asymp 1,
			\end{align*}
			for $M>1$, or in other words, neither \eqref{costhmsmall} nor \eqref{costhmbig} hold.
		\end{remark}
	\subsection{Proofs}
		
		\begin{proof}[Proof of Theorem~\ref{thm1}]
			The necessity part follows from the convergence at $r=0$ and the fact that $j_\alpha(0)=1$.
			
			In order to prove the sufficiency part, we show that the Cauchy remainder  \eqref{remainder} vanishes uniformly in $r$ as $N>M\to \infty$.
			
			Let $0<M<N$. If $r\geq 1/M$,  integration by parts yields
			\begin{align*}
			\int_M^N t^\nu f(t) j_\alpha(rt)\, dt = \big[ f(t) g^\nu_{\alpha,r}(t)\big]^N_M - \int_M^N g^\nu_{\alpha,r}(t)\, df(t).
			\end{align*}
			It follows from \eqref{Primest} that
			\begin{align*}	
			\bigg| \int_M^N t^\nu f(t) j_\alpha(rt)\, dt\bigg|&\lesssim \frac{1}{r^{\alpha+3/2}} \max_{x\geq M} x^{\nu-\alpha-1/2}|f(x)| + \frac{1}{r^{\alpha+3/2}}\int_M^N t^{\nu-\alpha-1/2}\, |df(t)| \\
			&	\leq \max_{x\geq M}x^{\nu+1}|f(x)| + M^{\alpha+3/2}\int_M^\infty t^{\nu-\alpha-1/2}\, |df(t)|,
			\end{align*}
			and both terms vanish as $M\to \infty$, by applying  \eqref{thm<}  and \eqref{thm>}.
			
			If $r<1/M$, we write 
			$$
			\int_M^N t^\nu f(t) j_\alpha(rt)\, dt = \bigg( \int_M^{1/r}+\int_{1/r}^N \bigg)  t^\nu  f(t) j_\alpha(rt)\, dt.
			$$
			The integral $\int_{1/r}^N t^\nu  f(t) j_\alpha(rt)\, dt$ can be estimated as above, as for the other integral, we have, by \eqref{smallest}
			\begin{align*}
			\bigg\vert \int_M^{1/r} t^{\nu} f(t) j_\alpha(rt)\, dt\bigg\vert &\leq  \max_{x>M} \bigg\vert \int_M^{x}  t^{\nu} f(t)\, dt \bigg\vert 	 + \bigg\vert \int_M^{1/r} t^{\nu} f(t) (1-j_{\alpha}(rt))\, dt\bigg\vert \\
			&\leq \max_{x>M} \bigg\vert \int_M^{x}  t^{\nu} f(t)\, dt\bigg\vert + r \int_M^{1/r} t^{\nu+1} |f(t)| rt\, dt \\
			& \leq \max_{x>M} \bigg\vert \int_M^{x}  t^{\nu} f(t)\, dt\bigg\vert + \bigg(\max_{x\geq M}x^{\nu+1}|f(x)|\bigg)  \int_M^{1/r}r\, dt \\
			&\leq \max_{x>M} \bigg\vert \int_M^{x}  t^{\nu} f(t)\, dt\bigg\vert + \max_{x\geq M}x^{\nu+1}|f(x)|.
			\end{align*}
			The first term of the latter inequality vanishes as $M\to \infty$ by \eqref{necsuf}, whilst the second term also vanishes as $M\to \infty$, by \eqref{thm<}.
		\end{proof}
		
		\begin{proof}[Proof of Theorem~\ref{costhmvanish}]
			Observe that since $f$ is vanishing at infinity, we have that $|f(x)|\leq \int_x^\infty |df(t)|$ for all $x$.
			
			Let us first consider the case $\nu<\alpha+1/2$. On the one hand,
			$$M^{\nu+1}|f(M)|\leq M^{\nu+1}\int_M^\infty |df(t)| \to 0 \quad \text{as }M\to \infty.$$
			On the other hand,
			$$
			M^{\alpha+3/2}\int_M^\infty t^{\nu-\alpha-1/2}|df(t)|\leq M^{\nu+1}\int_M^\infty |df(t)|\to 0 \quad \text{as }M\to \infty,
			$$
			and the result follows, since we are under the conditions of Theorem~\ref{thm1}. 
			
			If $\nu\geq \alpha+1/2$,
			$$M^{\nu+1}|f(M)|\leq M^{\nu+1}\int_M^\infty |df(t)|\leq M^{\alpha+3/2} \int_M^\infty t^{\nu-\alpha-1/2}|df(t)|\to 0 \quad \text{as }M\to \infty,$$
			i.e., we are under the conditions of Theorem~\ref{thm1}, and the result follows (notice that in this case \eqref{thm>} implies \eqref{thm<}).
			
			Finally, if $\nu\leq -1$, since $f$ vanishes at infinity, condition \eqref{thm<} is automatically satisfied, and the result follows, since \eqref{costhmbig} is precisely \eqref{thm>}.
		\end{proof}
		
		\begin{proof}[Proof of Corollary~\ref{CORcosgm}]			
			Similarly as above, the necessity follows from the convergence at $r=0$ and $j_\alpha(0)=1$.
			
			In order to prove the sufficiency part, we need the following result whose proof is rather technical and will be shown elsewhere for the sake of brevity:
			\begin{lemma}
				\label{THM-convergence-GM}
				Let $g\in GM$ be real-valued and assume that $\int_0^\infty g(t)\, dt$ converges. Then $tg(t)\to 0$ as $t\to \infty$.
			\end{lemma}
			The latter is a generalization of the well known Abel-Olivier's test that deals with nonnegative monotone functions (see also \cite{LTZ}, where the monotonicity assumption is relaxed). We emphasize that Theorem~\ref{THM-convergence-GM} $g$ only needs to be real-valued, instead of nonnegative. 
			
			Since $f\in GM$, it follows that $t^\nu f(t)\in GM$ for every $\nu \in \R$. Therefore, by Lemma~\ref{THM-convergence-GM} (with $g(t)=t^\nu f(t)$), the convergence of $\int_0^\infty t^\nu f(t)\, dt$ implies that $t^{\nu+1}f(t)\to 0$ as $t\to \infty$, which is precisely condition \eqref{thm<}.
			
			To conclude the proof, we show that if $f\in GM$, then \eqref{thm<} implies \eqref{thm>}, and the result will follow by Theorem~\ref{thm1}. Indeed, since $\alpha\geq -1/2$,
			\begin{align*}
			\int_M^\infty t^{\nu-\alpha-1/2}|df(t)|&\lesssim \int_{M/2}^\infty \frac{1}{t}\int_{t}^{2t}s^{\nu-\alpha-1/2}|df(t)|\lesssim \int_{M/2}^\infty t^{\nu-\alpha-5/2}\int_{t/\lambda}^{\lambda t}|f(s)|\, ds\\
			&\lesssim \int_{M/2}^\infty \bigg(\max_{t/\lambda\leq x\leq \lambda t} |f(x)|\bigg)t^{\nu-\alpha-3/2}\, dt\\
			 &\asymp \int_{M/2}^\infty \bigg(\max_{t/\lambda\leq x\leq \lambda t} x^{\nu+1}|f(x)|\bigg) t^{-\alpha-5/2}\, dt\\
			 &\lesssim \bigg(\max_{x\geq M/(2\lambda)} x^{\nu+1}|f(x)|\bigg)M^{-\alpha-3/2}.
			\end{align*}
			Thus, by \eqref{thm<},
			$$
			M^{\alpha+3/2}\int_M^\infty t^{\nu-\alpha-1/2}|df(t)|\lesssim \max_{x\geq M/(2\lambda)} x^{\nu+1}|f(x)| \to 0 \qquad \text{as }M\to \infty,
			$$
			i.e., \eqref{thm>} holds. This completes the proof.
		\end{proof}

		\begin{proof}[Proof of Theorem~\ref{thm1.5}]
			The necessity part is clear, due to the convergence at $r=0$. 
			
			Now we proceed to prove the sufficiency part. Let us denote
			$$
			F_\nu(x):=\int_x^\infty t^{\nu}f(t)\, dt.
			$$ 
			First of all, it follows from \eqref{est} and \eqref{dj2} that
			\begin{equation}
			\label{derivativeest}
			\bigg|\frac{d}{dt}j_\alpha(rt)\bigg| \lesssim \frac{1}{t^{\alpha+1/2}r^{\alpha-1/2}},
			\end{equation}
			whenever $rt\geq 1$, or equivalently, $r\geq 1/t$. Now we proceed to estimate the integral
			$$
			\int_M^\infty t^{\nu}f(t)j_\alpha(rt)\, dt,
			$$
			which is equivalent to estimate the Cauchy remainder \eqref{remainder} as $N\to \infty$.
			On the one hand, if $r\geq 1/M$, we integrate by parts and obtain
			\begin{align*}
			\bigg|	\int_M^\infty t^{\nu}f(t)j_\alpha(rt)\, dt \bigg| &\leq  \big| j_\alpha(rM)F_\nu(M)\big|+\bigg|\int_M^\infty F_\nu(t)\bigg(\frac{d}{dt}j_\alpha(rt)\bigg)\, dt\bigg|\\
			& \leq \max_{N\geq M}|F_\nu(N)|+ \max_{N\geq M}|F_\nu(N)|\int_M^\infty\bigg| \frac{d}{dt}j_\alpha(rt)\bigg|\, dt\\
			&\lesssim \max_{N\geq M}|F_\nu(N)|\bigg(1+\frac{1}{r^{\alpha-1/2}}\int_M^\infty \frac{1}{t^{\alpha+1/2}} \, dt\bigg)\\
			&\leq	\max_{N\geq M}|F_\nu(N)|\bigg(1+M^{\alpha-1/2}\int_M^\infty \frac{1}{t^{\alpha+1/2}} \, dt \bigg) \\
			&\asymp \max_{N\geq M}|F_\nu(N)|,
			\end{align*}
			where we have applied \eqref{derivativeest} and used the fact that $\alpha>1/2$. Since $F_\nu$ vanishes at infinity whenever \eqref{necsuf} is satisfied, the above estimate vanishes as $M\to \infty$. On the other hand, if $r<1/M$, we write
			$$
			\int_M^\infty t^{\nu}f(t)j_\alpha(rt)\, dt=\bigg(\int_M^{1/r}+\int_{1/r}^\infty \bigg) t^{\nu}f(t)j_\alpha(rt)\, dt,
			$$
			and estimate $\int_{1/r}^\infty 		t^{\nu}f(t)j_\alpha(rt)\, dt$ as in the previous case. Similarly as in the proof of Theorem~\ref{thm1}, estimate \eqref{smallest} yields
			\begin{align*}
			\bigg\vert \int_M^{1/r} t^{\nu} f(t) j_{\alpha}(rt) \, dt\bigg\vert &\leq \max_{x>M} \bigg|\int_M^x t^\nu f(t)\, dt\bigg| + r \int_M^{1/r} t^{\nu+1}|f(t)|\, dt \\
			& \leq \max_{x>M} \bigg|\int_M^x t^\nu f(t)\, dt\bigg| +\max_{x\geq M} x^{\nu+1}|f(x)|,
			\end{align*}
			which vanishes as $M\to \infty$.
		\end{proof}
	
	\section{Uniform convergence of $\mathcal{L}^\alpha_{\nu,\mu}f$ with $0<\mu+\nu \leq \alpha+3/2$}\label{sineu.c.}
	
	In this section we study the uniform convergence of sine-type transforms. We also mention some remarkable facts about the family of operators $\mathcal{L}^\alpha_{\alpha+1/2,0}$, $\alpha>-1/2$.

	\subsection{Main Results}
	
	Additionally to Theorem~\ref{sinethm}, here we give several results involving $GM$ functions, and in some cases we can obtain a criterion for the uniform convergence of $\mathcal{L}_{\nu,\mu}^\alpha f$. The extremal case $\mu+\nu=0$ is not mentioned here, since it is already treated in Section~\ref{secu.c.} (see Theorems~\ref{thm1}~and~\ref{costhmvanish}).
	
	\begin{remark}
		Let us observe an interesting property of the operator $\mathcal{L}^\alpha_{\alpha+1/2,0}$, with $\alpha>-1/2$ (if $\alpha=-1/2$, such operator is the cosine transform). Its kernel $K_\alpha(r,t) = K_\alpha(rt) := (rt)^{\alpha+1/2} j_\alpha(rt)$ is uniformly bounded and does not vanish at infinity in any of the variables $r$ nor $t$ (for any fixed $\alpha$, this is the only kernel of the type \eqref{generalkernel} with this property). Moreover, $K_\alpha$ vanishes at the origin. Thus, such kernel has a similar behaviour as the kernel $K_{1/2}(rt) =\sin rt$ corresponding to $\hat{f}_{\sin}$. In fact, more than extending the sine transform, the sufficient condition that guarantees the uniform convergence of $\mathcal{L}^{\alpha}_{\alpha+1/2,0}f$ and $\hat{f}_{\sin}$ is the same, namely (cf. Theorem~\ref{sinethm})
	\begin{equation*}
	\int_x^\infty |df(t)|=o(1/x) \quad \text{as }x\to \infty.
	\end{equation*}
	\end{remark}
	
	Similarly as for cosine-type transforms, in Theorem~\ref{sinethm} we do not assume that $f$ vanishes at infinity; for functions satisfying the latter we claim the following:
		\begin{theorem}
			\label{sinethmvanish}
			Let $\nu,\mu\in \R$ be such that $0<\mu+\nu\leq \alpha+3/2$, and let $f$ be vanishing at infinity and such that $t^\nu f(t)\in L^1(0,1)$. Assume that
			\begin{align}
			\int_M^\infty |df(t)| &= o\big( M^{\mu-1}\big) \quad \text{as }M\to \infty, & \text{if }\nu<\alpha+1/2 \text{ and }\mu< 1, \label{sinthmsmall}\\
			\int_M^\infty  t^{\nu-\alpha-1/2}|df(t)| &=o\big(M^{\mu+\nu-\alpha-3/2}\big) \quad \text{as }M\to \infty, &\text{if }\nu\geq \alpha+1/2 \text{ or }\mu\geq 1.\label{sinthmbig}
			\end{align}
			Then $\mathcal{L}_{\nu,\mu}^\alpha f$ converges uniformly on $\R_+$.
		\end{theorem}

	We can refine Theorem~\ref{sinethm} by assuming that $f\in GM$. Furthermore, in this case we can obtain a criterion for non-negative $GM$ functions.
	
		\begin{theorem}\label{sinegm}
			Let $\nu,\mu$ be such that $0<\mu+\nu <\alpha+3/2$. Let $f$ be a $GM$ function such that $t^\nu f(t)\in L^1(0,1)$.
			\begin{enumerate}
				\item If 
				\begin{equation}
				\label{cond1-aux}
				|f(M)|= o\big(M^{\mu-1}\big)\quad \text{as }M\to \infty,
				\end{equation}
				then $\mathcal{L}_{\nu,\mu}^\alpha f$ converges uniformly.
				\item If $f\geq 0$ and $\mathcal{L}_{\nu,\mu}^\alpha f$ converges uniformly, then \eqref{cond1-aux} holds.
			\end{enumerate}
		\end{theorem}
		The ``if and only if'' statement reads as follows:
		\begin{corollary}\label{corgmsine}
			Let $f\in GM$ be non-negative, $\alpha\geq -1/2$, and $\nu,\mu$ be such that $0<\mu+\nu<\alpha+3/2$. Then, $\mathcal{L}_{\nu,\mu}^\alpha f$ converges uniformly if and only if \eqref{cond1-aux} holds.
		\end{corollary}
		
		The $GM$ condition in the sufficiency part of Theorem~\ref{sinegm} (and therefore also in Corollary~\ref{corgmsine}) is sharp, as shown by our next statement.
		\begin{proposition}
			\label{PROPsharpgm}
			Let $0<\mu+\nu<\alpha+3/2$. There exists $f\not\in GM$ such that condition \eqref{cond1-aux} does not hold, but $\mathcal{L}^\alpha_{\nu,\mu}f$ converges uniformly.
		\end{proposition}

		Note that in Theorem~\ref{sinegm} we exclude the case $\mu+\nu=\alpha+3/2$. Actually, the proof of the latter relies on the fact that if $f\in GM$ and $0<\mu+\nu<\alpha+3/2$, then \eqref{cond1-aux} implies \eqref{sinthmbig}, and the result follows by Theorem~\ref{sinethm}. This is not the case in the extremal case $\mu+\nu=\alpha+3/2$.

		\begin{proposition}\label{sharpparams2}
			Let $\nu\in \R$ and $\mu<1$ be such that $\mu+\nu=\alpha+3/2$. If $f\in GM$ vanishes at infinity, then \eqref{sinthmbig} is equivalent to $t^{-\mu}f(t)\in L^1(1,\infty)$.
		\end{proposition}
		It is clear that if $\mu+\nu=\alpha+3/2$, then \eqref{cond1-aux} does not imply \eqref{sinthmbig}, since the latter does not imply $t^{-\mu}f(t)\in L^1(1,\infty)$ even for decreasing functions. Moreover, Proposition~\ref{sharpparams2} does not hold for $\mu=1$. Indeed, for $f$ decreasing we have
		$$
		\int_M^\infty t^{1-\mu}|df(t)|=\int_M^\infty |df(t)|=f(M),
		$$
		and the convergence of $\int_1^\infty |df(t)|$ is equivalent to $f(M)\to 0$ as $M\to \infty$. As mentioned above, such condition does not imply that $M^{-1}f(M)$ is integrable.

\subsection{Proofs}

	\begin{proof}[Proof of Theorem~\ref{sinethm}]
		Again, we prove that the Cauchy remainder \eqref{remainder} vanishes uniformly in $r$ as $N>M\to \infty$. 		Let $0<M<N$, and assume that $1/r\leq M$. Integration by parts together with the representation of \eqref{primformula}, and  estimate \eqref{Primest}  yield
		\begin{align*}
		r^{\mu+\nu}\int_M^N t^{\nu} f(t) j_{\alpha}(rt)\, dt& =r^{\mu+\nu} \bigg( \big[f(t) g^{\nu}_{\alpha,r}(t) \big]_M^N -\int_M^N g^{\nu}_{\alpha,r}(t)\, df(t) \bigg)  \\
		& \lesssim r^{\mu+\nu} \bigg( \frac{N^{\nu-\alpha-1/2}}{r^{\alpha+3/2}}|f(N)|+ \frac{M^{\nu-\alpha-1/2}}{r^{\alpha+3/2}}|f(M)|  \\
		&\phantom{=} + \frac{1}{r^{\alpha+3/2}} \int_M^{N}t^{\nu-\alpha-1/2}\, |df(t)|\bigg) \\
		&\lesssim \max_{x\geq M}x^{1-\mu}|f(x)| + M^{\alpha+3/2-\mu-\nu}\int_M^\infty t^{\nu-\alpha-1/2}\,|df(t)|,
		\end{align*}
		which vanishes as $M\to \infty$, by \eqref{cond1-aux} and \eqref{sinthmbig}.
		
			If $1/r>M$, we write
			$$
			r^{\mu+\nu}\int_M^N t^{\nu} f(t) j_{\alpha}(rt)\, dt = r^{\mu+\nu} \bigg(\int_M^{1/r}+ \int_{1/r}^N\bigg) t^\nu f(t) j_\alpha(rt)\, dt,
			$$
			and estimate the integral $\int_{1/r}^N t^\nu f(t)j_\alpha(rt)\, dt$ as above. Furthermore, since $\mu+\nu >0$, it follows that
			\begin{align*}
			\bigg| r^{\mu}\int_M^{1/r} (rt)^\nu f(t) j_\alpha(rt)\, dt\bigg|&\leq  r^{\mu} \int_{M}^{1/r} (rt)^\nu |f(t)|\, dt = r^{\mu+\nu} \int_M^{1/r}t^{\nu}|f(t)|\, dt \\
			&= r^{\mu+\nu}\int_M^{1/r} t^{\mu+\nu-1} t^{1-\mu}|f(t)|\, dt \\
			&\leq\bigg( \max_{x\geq M}x^{1-\mu}|f(x)|\bigg) r^{\mu+\nu}\int_0^{1/r}t^{\mu+\nu-1}\, dt \\
			&\asymp \max_{x\geq M}x^{1-\mu}|f(x)|,
			\end{align*}
			which vanishes as $M\to \infty$, by \eqref{cond1-aux}.
			\end{proof}

\begin{proof}[Proof of Theorem~\ref{sinethmvanish}]
		We will see that our hypotheses imply those of Theorem~\ref{sinethm}, and the result will follow. Consider first the case $\mu<1$ and $\nu<\alpha+1/2$. Then
		$$
		M^{1-\mu}|f(M)|\leq M^{1-\mu}\int_M^\infty|df(t)| \to 0 \quad \text{as }M\to \infty,
		$$
		and
		$$
		M^{\alpha+3/2-\mu-\nu}\int_M^\infty t^{\nu-\alpha-1/2}|df(t)|\leq M^{1-\mu}\int_M^\infty|df(t)|
	\to 0 \quad \text{as }M\to \infty.
		$$
		
		If $\mu \geq 1$, since $f$ vanishes at infinity, \eqref{cond1} holds, and the hypotheses of Theorem~\ref{sinethm} are met.
		
		Finally, if $\nu \geq \alpha+1/2$,
		$$
		M^{1-\mu}|f(M)|\leq M^{1-\mu}\int_M^\infty|df(t)| \leq M^{\alpha+3/2-\mu-\nu}\int_M^\infty t^{\nu-\alpha-1/2} |df(t)| \to 0 \quad \text{as }M\to \infty,
		$$
		i.e., the hypotheses of Theorem~\ref{sinethm} hold.
\end{proof}

		\begin{proof}[Proof of Theorem~\ref{sinegm}]
			Since $f\in GM$, \eqref{cond1-aux} implies \eqref{sinthmbig} for any choice of the parameters. Indeed, since $\mu+\nu<\alpha+3/2$,
			\begin{align*}
			\int_M^\infty t^{\nu-\alpha-1/2} |df(t)| &\lesssim \int_{M/2}^\infty \frac{1}{t}\int_t^{2t} s^{\nu-\alpha-1/2}|df(s)| \lesssim \int_{M/2}^\infty t^{\nu-\alpha-5/2}\int_{t/\lambda}^{\lambda t}|f(s)|\, ds\\
			& \lesssim \int_{M/2}^\infty \bigg(\max_{t/\lambda\leq x\leq \lambda t} |f(x)| \bigg)t^{\nu-\alpha-3/2} \, dt \\
			&\asymp \int_{M/2}^{\infty} \bigg(\max_{t/\lambda\leq x\leq \lambda t} x^{1-\mu}|f(x)| \bigg)t^{\mu+\nu-\alpha-5/2}\, dt \\ &\lesssim \bigg( \max_{x\geq M/(2\lambda)} x^{1-\mu}|f(x)|\bigg) M^{\mu+\nu-\alpha-3/2}.
			\end{align*}
			Thus, we deduce that 
			$$
			M^{\alpha+3/2-\mu-\nu}\int_M^\infty t^{\nu-\alpha-1/2}\, |df(t)|\lesssim  \max_{t\geq M/(2\lambda)} t^{1-\mu}|f(t)| \to 0 \quad \text{as }M\to \infty,
			$$
		so that the result follows by applying Theorem~\ref{sinethm}.  This completes the first part of the proof.
		
		For the second part, the uniform convergence of $\mathcal{L}_{\nu,\mu}^\alpha f$ implies that the Cauchy remainder
		$$
		r^{\mu}\int_{1/(\lambda r)}^{\lambda/r} (rt)^\nu f(t)j_{\alpha}(rt)\, dt \asymp r^{\mu}\int_{1/(\lambda r)}^{\lambda/r}  f(t)  \, dt
		$$
		vanishes whenever $r\to 0$, where $\lambda>0$ is the $GM$ constant (cf. Definition~\ref{defgm}). By \eqref{gmest}, we have
		$$
		f(1/r) \lesssim r\int_{1/(\lambda r)}^{\lambda /r} f(t)\, dt =r^{1-\mu} r^{\mu}\int_{1/( \lambda r)}^{\lambda/r} f(t)\, dt,
		$$
		and we deduce that $r^{\mu-1}f(1/r)\to 0$ as $r\to 0$, or equivalently, $t^{1-\mu}f(t)\to 0$ as $t\to \infty$.
		\end{proof}

	\begin{proof}[Proof of Proposition~\ref{PROPsharpgm}]
		We construct $f$ in a general setting and then we subdivide the proof into two parts, namely $0<\mu+\nu\leq \alpha+1/2$ and $\alpha+1/2<\mu+\nu<\alpha+3/2$.
		
		Let $c_n$ be an increasing nonnegative sequence and $\varepsilon_n>0$ such that $\varepsilon_n<c_{n+1}-c_n$ and $\varepsilon_n\leq c_n$ for every $n$. Define
		$$
		f(t)=\begin{cases}
		t^{\mu-1}, &\text{if }t\in [c_n,c_n+\varepsilon_n],\, n\in \N,\\
		0,& \text{otherwise.}
		\end{cases}
		$$
		It is clear that for such function, $t^{1-\mu}f(t)\not \to 0$ as $t\to \infty$. We are now going to find choices of $c_n$ and $\varepsilon_n$ in such a way that $f\not\in GM$ and  $\mathcal{L}^\alpha_{\nu,\mu}f$ converges uniformly. Also, for any $c_n$ and $\varepsilon_n$, $t^\nu f(t)\in L^1(0,1)$, since $\mu+\nu>0$.
		
		Let us first consider the case $0<\mu+\nu\leq \alpha+1/2$. According to Corollary~\ref{PROPuclacunary}, the uniform convergence of $\mathcal{L}^\alpha_{\nu,\mu}f$ follows from $t^\nu f(t)\in L^1(\R_+)$, which in this case is equivalent to
		$$
		\sum_{n=1}^\infty \varepsilon_n c_n^{\nu+\mu-1}<\infty.
		$$
		Choosing $c_n=2^{n}$ and $\varepsilon_n=2^{-n\beta}$ with $\beta>\nu+\mu-1$, we find that the latter series converges, hence the uniform convergence of $\mathcal{L}^\alpha_{\nu,\mu}f$ follows. Note also that $f\not\in GM$.
		
		Consider now the case $\alpha+1/2<\mu+\nu<\alpha+3/2$.  According to Corollary~\ref{roughcorol}, the uniform convergence of $\mathcal{L}^\alpha_{\nu,\mu}f$ follows from the conditions
		$$
		t^{\nu-\alpha-1/2}f(t)\to 0 \text{ as }t\to \infty, \qquad t^{\nu-\alpha-1/2}f(t)\in L^1(1,\infty),\qquad \int_1^\infty t^{\nu-\alpha-1/2}|df(t)|<\infty.
		$$
		Since $\mu+\nu<\alpha+3/2$, $t^{\nu-\alpha-1/2}f(t)\to 0$ as $t\to\infty$. Also,
		\begin{equation}
		\label{EQintegrab-counterex}
		\int_1^\infty t^{\nu-\alpha-1/2}f(t) \, dt =\sum_{n=1}^\infty \varepsilon_n c_n^{\mu+\nu-\alpha-3/2},
		\end{equation}
		and
		\begin{equation}
		\label{EQvariation-counterex}
		\int_1^\infty t^{\nu-\alpha-1/2}|df(t)|\lesssim \sum_{n=1}^\infty \Big( c_n^{\mu+\nu-\alpha-3/2}+(c_n+\varepsilon_n)^{\mu+\nu-\alpha-3/2} \Big) \lesssim 
		\sum_{n=1}^\infty c_n^{\mu+\nu-\alpha-3/2}.
		\end{equation}
		Choosing $c_n=2^n$ and $\varepsilon_n=1$, we find that series on the right hand sides of \eqref{EQintegrab-counterex} and \eqref{EQvariation-counterex} are convergent, so that $\mathcal{L}^\alpha_{\nu,\mu}f$ converges uniformly by Corollary~\ref{roughcorol}, and $f\not\in GM$.
	\end{proof}
	
	\begin{proof}[Proof of Proposition~\ref{sharpparams2}]
		The proof is similar to that of Proposition 5.4 in \cite{LTZ}. First of all note that since $f$ is locally of bounded variation, condition \eqref{sinthmbig} is equivalent to the convergence of $\int_1^\infty t^{1-\mu} |df(t)|$. Since $f$ vanishes at infinity and $\mu<1$, the estimate
		\begin{align*}
		\int_1^\infty t^{-\mu}|f(t)|\, dt &\leq \int_1^\infty t^{-\mu}\int_t^\infty |df(s)|\, dt=\int_1^\infty |df(s)|\int_1^s t^{-\mu}\, dt \lesssim \int_1^\infty t^{1-\mu}|df(t)|
		\end{align*}
		proves one direction of the statement, without mention to the $GM$ condition. As for the other direction, we have, since $f\in GM$,
		\begin{align*}
		\int_{1}^\infty s^{1-\mu}|df(s)| &=\sum_{k=0}^\infty \int_{2^k}^{2^{k+1}} s^{1-\mu}|df(s)|\lesssim \sum_{k=0}^\infty (2^k)^{1-\mu} \frac{1}{2^k}\int_{2^k/\lambda}^{\lambda 2^k}|f(t)|\, dt \\
		&\asymp \sum_{k=0}^\infty \int_{2^k/\lambda}^{\lambda 2^k} t^{-\mu}|f(t)|\, dt \lesssim \int_{1/\lambda}^\infty t^{-\mu}|f(t)|\, dt,
		\end{align*}
		as desired. Observe that the latter holds for any $\mu$.
	\end{proof}	
	
	\subsection{Optimality of Theorems~\ref{sinethm} and \ref{sinethmvanish}}
	
	\textbf{Sharpness}. Here we are interested in studying if the conclusions of Theorems~\ref{sinethm}~and~\ref{sinethmvanish} hold if we replace $o$ by $O$ in conditions \eqref{cond1} and \eqref{cond2}, or \eqref{sinthmsmall} and \eqref{sinthmbig}.
	
	 1. Case $\mu< 1$. In this case, we will not discuss sharpness of Theorem~\ref{sinethm}, since condition \eqref{cond1} implies that $f$ vanishes at infinity, and therefore we are in the situation of Theorem~\ref{sinethmvanish}. Consider the function $f(t)=t^{1-\mu}$ and $\mu+\nu<\alpha+3/2$. It is clear that neither \eqref{sinthmsmall} nor \eqref{sinthmbig} hold, but they are satisfied if we replace $o$ by $O$. Since $\mu+\nu>0$, we have for any $r>0$
	$$
	r^{\mu+\nu}\int_{1/(2r)}^{1/r}t^{\nu} f(t) j_\alpha(rt)\, dt \asymp r^{\mu+\nu}\int_{1/(2r)}^{1/r} t^{\mu+\nu-1}\, dt\asymp 1,
	$$
	i.e., the Cauchy remainder does not vanish as $r\to 0$, and therefore $\mathcal{L}^\alpha_{\nu,\mu} f$ does not converge uniformly.

	2. Case $\mu=1$. Note that in this case the statements of Theorems~\ref{sinethm}~and~\ref{sinethmvanish} are equivalent. If $f(t)=1$, it is clear that \eqref{cond1} does not hold, but holds with $O$ in place of $o$, whilst \eqref{cond2} trivially holds. The Cauchy remainder is the same as in the previous example, substituting $\mu=1$, and thus $\mathcal{L}^\alpha_{\nu,\mu}f$ does not converge uniformly. 

	3. Case $\mu>1$. Here the example $f(t)=t^{1-\mu}$ shows that Theorem~\ref{sinethm} does not hold if we replace $o$ by $O$ in \eqref{cond1} and \eqref{cond2}. The examples $f(t)=t^{\mu-2}\sin t$ and $f(t)=1$ show that in general, conditions \eqref{cond1} and \eqref{cond2} do not imply each other.

	Finally, we show that the sufficient conditions involving the variation of $f$ that imply the uniform convergence of $\mathcal{L}^\alpha_{\nu,\mu}f$ (see Theorems~\ref{sinethm}~and~\ref{sinethmvanish}) do not imply neither follow from those integrability conditions that also imply the uniform convergence of $\mathcal{L}^\alpha_{\nu,\mu}f$ (cf. Corollary~\ref{roughcorol}).
		\newline
	
	\noindent\textbf{Independence of Theorem~\ref{sinethm}~and Corollary~\ref{roughcorol}}. 	
	Let $f(t)=t^{\mu-2}\sin t$ for $t>1$, and $\alpha+1/2\leq \mu+\nu<\alpha+3/2$. Since
	$$
	f'(t)=(\mu-2)t^{\mu-3}\sin t+t^{\mu-2}\cos t,
	$$
	we have that
	$$
	M^{\alpha+3/2-\mu-\nu}\int_M^\infty t^{\nu-\alpha-1/2}|f'(t)|\, dt \asymp M^{\alpha+3/2-\mu-\nu}\int_M^\infty t^{\mu+\nu-\alpha-5/2}\, dt \asymp 1,
	$$
or in other words, \eqref{cond2} does not hold. Thus, the hypotheses of Theorem~\ref{sinethm} are not satisfied. Nevertheless, the choice of the parameters implies that $t^{\nu-\alpha-1/2}f(t)\in L^1(0,1)$ (and $t^{-\mu}f(t)\in L^1(1,\infty)$ if $\mu+\nu=\alpha+1/2$), and moreover the conditions \eqref{pointwisecond} hold.  Hence, the uniform convergence of $\mathcal{L}^\alpha_{\nu,\mu}f$ follows, by Corollary~\ref{roughcorol}.

On the other hand, let $\alpha=1/2$, $\nu=1$ and $\mu=0$ (recall that $\mathcal{L}^{1/2}_{1,0} f = \hat{f}_{\sin}$). If $f(t)=\dfrac{1}{t\log t}$ for $t>2$, then clearly $f(t)\not\in L^1(2,\infty)$, but \eqref{cond2} holds, and the uniform convergence of $\hat{f}_{\sin}$ follows by Theorem~\ref{sinethm} (or also by Theorem~\ref{sinegm}, since $f$ vanishes at infinity).
\newline

\noindent\textbf{Independence of Theorem~\ref{sinethmvanish} and Corollary~\ref{roughcorol}}. Let us consider again $f(t)=t^{\mu-2}\sin t$, with $\mu<2$. We have already seen that $M^{\alpha+3/2-\mu-\nu}\int_M^\infty t^{\nu-\alpha-1/2}|f'(t)|\, dt \asymp 1$, and that additionally to $t^{\nu-\alpha-1/2}f(t)\in L^1(1,\infty)$, the conditions \eqref{pointwisecond} hold. Thus, in the case $\nu\geq \alpha+1/2$ or $\mu\geq 1$, we cannot apply Theorem~\ref{sinethmvanish}, but we can apply Corollary~\ref{roughcorol} instead to deduce the uniform convergence of $\mathcal{L}^\alpha_{\nu,\mu}f$. On the other hand, it is easy to see that if $\nu<\alpha+1/2$ and $\mu<1$, the hypotheses of Corollary~\ref{roughcorol} hold, and $M^{1-\mu}\int_M^{\infty}|f'(t)|\, dt \asymp 1$.

Now let $\nu\geq \alpha+1/2$ and $\alpha+1/2-\nu\leq \mu<1$. If $f(t)=\dfrac{t^{\alpha-\nu-1/2}}{\log t}$, then $f$ vanishes at infinity, but $t^{\nu-\alpha-1/2}f(t)=\dfrac{1}{t\log t}\not\in L^1(2,\infty)$. However,
$$
M^{\alpha+3/2-\mu-\nu}\int_M^\infty t^{\nu-\alpha-1/2}|f'(t)|\, dt \lesssim M\int_M^\infty  \frac{1}{t^2\log t}\, dt \lesssim \frac{1}{\log M}\to 0 \quad \text{as }M\to \infty,
$$
so that \eqref{sinthmbig} holds. In the case $\mu\geq 1$, note that $\nu<\alpha+1/2$, and hence the inequality $\mu+\nu\geq \alpha+1/2$ implies that $\alpha-1/2\leq \nu$. Thus, the same function $f$ as above vanishes at infinity, and also satisfies \eqref{sinthmbig}, whilst $t^{\nu-\alpha-1/2}f(t)\not\in L^1(2,\infty)$. Finally, consider the case $\nu<\alpha+1/2$ and $\mu<1$. Let $f(t)=\dfrac{t^{\mu-1}}{\log t}$. The inequality $\mu+\nu\geq \alpha+1/2$ implies that
$$
t^{\nu-\alpha-1/2} f(t) =\dfrac{t^{\mu+\nu-\alpha-3/2}}{\log t}\geq \frac{1}{t\log t} \not\in L^1(2,\infty),
$$
hence $f$ is not under the hypotheses of Corollary~\ref{roughcorol}. However, note that since $f$ is monotone,
$$
M^{1-\mu}\int_M^\infty |f'(t)|\, dt = M^{1-\mu}f(M) =\frac{1}{\log M}\to 0 \quad \text{as }M\to \infty,
$$
and $\mathcal{L}^\alpha_{\nu,\mu}f$ converges uniformly, in virtue of Theorem~\ref{sinethmvanish} (or also by Theorem~\ref{sinegm}).

	\end{document}